\documentclass[11pt]{amsart}
\usepackage[utf8]{inputenc} 
\usepackage[a4paper]{geometry} 
\usepackage{enumerate} 
\usepackage[all,cmtip]{xy}
\usepackage{hyperref}  
                 \hypersetup{ pdfborder={0 0 0}, 
                              colorlinks=true, 
                              linktoc=page,
                              pdfauthor={H. Awada, M. Bolognesi, and G. Staglian\`o},
                              pdftitle={Some Moduli of n-pointed Fano fourfolds}
                            }
\usepackage{verbatim}

\theoremstyle{plain} 
\newtheorem{proposition}{Proposition}[section] 
\newtheorem{theorem}[proposition]{Theorem} 
 
\newtheorem{corollary}[proposition]{Corollary} 
\theoremstyle{definition} 
\newtheorem{definition}[proposition]{Definition} 
 
\theoremstyle{remark} 
\newtheorem{remark}[proposition]{Remark}

\newcommand{\PP}{{\mathbb{P}}}  
\newcommand{\YY}{{\mathbb{Y}^5}}  
\newcommand{\MM}{{\mathcal M}_{GM}^4}
\newcommand{\MMo}{{\mathring{\mathcal M}}_{GM}^4}

\newcommand{\Vtre}{\mathcal{V}^{nod}_3}
\newcommand{\Vdue}{\mathcal{V}^{nod}_2}

\newcommand{\Vn}{\mathcal{V}^{nod}_n}

\newcommand{\C}{{\mathcal{C}}}

\numberwithin{equation}{section}

\title[Some Moduli of $n$-pointed Fano fourfolds]{Some Moduli of $n$-pointed Fano fourfolds}

\author[H. Awada]{Hanine Awada}
\address{Institut Montpellierain Alexander Grothendieck \\ %
Universit\'e de Montpellier \\ CNRS %
Case Courrier 051 - Place Eug\`ene Bataillon \\ %
34095 Montpellier Cedex 5 \\ %
France}
\email{hanine.awada@umontpellier.fr}

\author[M. Bolognesi]{Michele Bolognesi}
\address{Institut Montpellierain Alexander Grothendieck \\ %
Universit\'e de Montpellier \\ CNRS %
Case Courrier 051 - Place Eug\`ene Bataillon \\ %
34095 Montpellier Cedex 5 \\ %
France}
\email{michele.bolognesi@umontpellier.fr}

\author[G. Staglianò]{Giovanni Staglian\`o}
\address{Dipartimento di Matematica e Informatica  \\ %
Universit\`a degli Studi di Catania  \\ %
Viale A. Doria 5 - 95125 Catania  \\ %
Italy}
\email{giovanni.stagliano@unict.it}

\begin{document}

\begin{abstract}
The object of this note is the moduli spaces of cubic fourfolds (resp., Gushel-Mukai fourfolds) which contain some special rational surfaces. Under some 
hypotheses on the families of such surfaces, 
we develop a general method to show the unirationality of the moduli spaces of the $n$-pointed such fourfolds. We apply this to some codimension 1 loci of cubic fourfolds (resp., Gushel-Mukai fourfolds) appeared in the literature recently.
\end{abstract}

\maketitle 

\section{Introduction}
One of the most active areas of research in algebraic geometry is related to the study of the birational geometry of Fano varieties, notably those of dimension four. In the last 20 years,  algebraic geometers have been working on the problem of rationality of smooth cubic hypersurfaces in $\PP^5$(cubic fourfolds for short).
Recall that cubic fourfolds are parametrized by an open subset $\mathcal U$
in the $55$-dimensional projective space $\PP(\mathcal O_{\PP^5}(3))$. The moduli space 
of cubic fourfolds is the GIT quotient $\mathcal C = \mathcal U//\mathrm{PGL}_6$,
a quasi-projective variety of dimension $55-35=20$.
It is classically well-known that all cubic fourfolds are unirational and that some of them are rational. While the general suspicion is that most cubic fourfolds are non rational, no cubic fourfold has yet been proven to be non rational. 

Hassett, in his works \cite{Hassett,Has00} (see also \cite{Levico}), adopted a Hodge theoretic approach while studying cubic fourfolds.  He defined
the \textit{Noether-Lefschetz} locus as the subset of the moduli space $\mathcal C$ consisting of the \emph{special} cubic fourfolds, 
that is, fourfolds $X$ containing an algebraic surface $S$ that is not homologous to a complete intersection. 
One says that $X$ has discriminant $d$, which is defined as
 the determinant of the intersection form on the saturated sublattice of $H^{2,2}(X,\mathbb{Z})$ generated by $h^2$ and $[S]$,
where $h$ denotes the hyperplane section class of $X$.
Using the period map and the geometry of the period domain, Hassett proved that special cubic fourfolds form a countably infinite union of irreducible divisors $\C_d\subset \mathcal C$, corresponding to fourfolds having discriminant $d$, and where $d$ runs over all integers 
$d \geq 8$ with $d \equiv 0,2 \ (\mathrm{mod}\ 6)$. 
For small values of the discriminant $d$, these divisors are characterized by the
families of surfaces, not unique, that they contain (see \cite{Has00, Nuer, RS1, RS3}). Moreover, for an infinite values of $d$, cubic fourfolds in $\C_d$ are associated to a degree $d$ polarized K3 surface via Hodge theory. This seems to relate strongly to the rationality of cubic fourfolds. In fact, it is conjectured that fourfolds with an associated K3 surface should be precisely the rational ones (see \cite{kuz4fold,AT,kuz2,Levico,BRS,RS1,RS3}).

Another class of Fano fourfolds has emerged also: Gushel-Mukai fourfolds (GM fourfolds, for short), prime Fano fourfolds of degree 10 and index 2. By a result of Mukai \cite{mukai-biregularclassification}, they can be realized as smooth dimensionally transverse intersection of a cone $C(\mathbb{G}(1,4))\subset\PP^{10}$ 
 over the Grassmannian $\mathbb{G}(1,4)\subset\PP^9$ with a linear subspace $\mathbb{P}^8\subset\PP^{10}$ and a quadric hypersurface $Q\subset\PP^{10}$. 
The fourfolds for which the $\PP^8\subset\PP^{10}$ does not pass by the vertex of the cone
 are called \emph{ordinary}. These can be viewed as smooth quadric hypersurfaces in a smooth del Pezzo fivefold 
 $\YY=\mathbb{G}(1,4)\cap \PP^8\subset\PP^8$, thus parametrized 
 by an open subset $\mathcal V$ in the $39$-dimensional projective space $\mathbb{P}(\mathcal O_{\YY}(2))$;
 recall also that all such $\YY$
are projectively equivalent. The moduli space of GM fourfolds has dimension $24$
 and is denoted by $\MM$ (see \cite{DIM}). 
The ordinary GM fourfolds correspond to the points of an open subset $\MMo$ in $\MM$, 
 which  is the complementary of an irreducible closed subset of codimension $2$ in $\MM$.
 We can view $\MMo$ as the quotient $\mathcal V//\mathrm{PGL}_9$.

Under the point of view of birational geometry, GM fourfolds behave very much like cubic fourfolds and share many properties with them. They are again all unirational, rational examples are easy to construct, but no examples have yet been proven to be nonrational. Once again, by the study of the period map via Hodge theory, in \cite{DIM} the authors introduced the \textit{Noether-Lefschetz} locus inside the  moduli space $\MM$, 
 defined as the set of those fourfolds 
 containing a surface whose cohomology class does not come from the Grassmannian $\mathbb{G}(1,4)$.
This locus 
 consists of 
 a countable infinite union of divisors $(\MM)_d\subset\MM$, labelled 
 by the integers $d>8$ with $d\equiv0,2,4\ (\mathrm{mod}\ 8)$. The divisor 
 $(\MM)_d$ is  irreducible if $d\equiv 0,4 \ (\mathrm{mod} \ 8)$, 
 and it has two irreducible components $(\MM)_d'$ and $(\MM)_d^{''}$ if $d \equiv 2 \ (\mathrm{mod} \ 8)$; 
 see \cite{DIM,DK3}. 
 Recently, the third-named author  \cite{famGushelMukai} (see also \cite{HoffSta}), inspired by the work of Nuer \cite{Nuer}, gave an explicit description
 of the first irreducible components of this Noether-Lefschetz locus 
 in terms of classes of rational smooth surfaces  that the fourfolds have to contain.

Lately, the first and second named authors were interested in these objects, but under a slightly different point of view. In a fashion very similar to curves and K3 surfaces \cite{farkas2019unirationality, FarkasVerraC42, ma2019mukai, Barros_2018}, universal families were defined on cubic fourfolds \cite{awada2020unirationality}
. Since a generic cubic fourfold in any divisor $\C_d$ doesn't have projective automorphism,  universal cubic fourfolds $\C_{d,1} \rightarrow  \C_d$  were introduced over divisors for $8 \leq d \leq 42$. These universal cubic fourfolds $\C_{d,1}$ correspond to the moduli space of 1-pointed cubic fourfolds.
The authors prove the unirationality of $\C_{d,1}$ for $8 \leq d \leq 42$, using the presentation of the divisors $\C_d$ as cubics containing certain rational surfaces (see \cite{Nuer, RS3}). This, combined with a theorem of Koll\'ar \cite{kollar} on the unirationality of smooth cubic fourfolds over an arbitrary field, gives the result. Inductively, they prove the unirationality of $\C_{d,n}$ for the same range of values of $d$, for all $n$ \cite[Theorem 4.10]{awada2020unirationality}.


\medskip
In this paper we propose a unified, general, abstract method to show the unirationality of the $n$-pointed universal (cubic and GM) fourfolds over their moduli spaces. This can be applied to any family $\mathcal{X}_\mathcal{S}$ of (cubic and GM) fourfolds that contains surfaces from a given family $\mathcal{S}$, under some hypotheses (see Rmk. \ref{properties}) on $\mathcal{X}_\mathcal{S}$ and $\mathcal{S}$. 

\smallskip

\begin{theorem}
The universal $n$-pointed fourfold over the following irreducible codimension-one loci:
\[\C_{14},\ \C_{26},\ \C_{38},\ \C_{42},\ (\MM)_{10}',\ (\MM)_{10}^{''},\ (\MM)_{20},\]
are unirational.
\end{theorem}

In the last section of the paper we restrict our attention to a codimension one locus $(\mathcal{M}^4_{GM})_{20}^{nod}$ inside $(\MM)_{20}$, defined via certain genus 11 K3 surfaces contained in a Noether-Lefschetz divisor. By describing the birational geometry of the NL divisor in the moduli of K3 surfaces, and exploiting the relation between these surfaces and the GM fourfolds, we prove that $(\mathcal{M}^4_{GM})_{20}^{nod}$ and the universal family $(\mathcal{M}^4_{GM})_{20,1}^{nod}$ above it are rational.

\bf Plan of the paper: \rm in Section \ref{famiglie di superfici}, we give explicit descriptions of certain divisors parametrizing (cubic or GM) fourfolds in their moduli space. We recall the constructions of several families $\mathcal{S}$ of surfaces characterizing these divisors and highlight some of their properties crucial for the next section. Section \ref{mainthm} is devoted to the proof of the main result of this paper. Finally, in Section \ref{g11} we describe the locus $(\mathcal{M}^4_{GM})_{20}^{nod} \subset (\MM)_{20}$ and show its rationality as well as that of the unviersal family $(\mathcal{M}^4_{GM})_{20,1}^{nod}$.

\section{Explicit geometric descriptions of some Nother-Lefschetz divisors in 
the moduli space $\mathcal C$ of cubic fourfolds 
and in the moduli space $\MM$ of GM fourfolds}\label{famiglie di superfici}

In this section, we shall recall some explicit descriptions 
of unirational irreducible families $\mathcal S$ in the Hilbert scheme of $\mathbb{P}^5$ 
(respectively in the Hilbert scheme of a fixed smooth del Pezzo fivevold $\YY=\mathbb{G}(1,4)\cap\mathbb{P}^8\subset\mathbb{P}^8$)
such that the closure of the locus of cubic fourfolds (resp., GM fourfolds) 
containing a surface of the family $\mathcal S$
describes a Noether-Lefschetz divisor in the corresponding moduli space. 
We shall focus on the fact that starting from a pair 
 $(S,X)$, where 
$S$ is a general member of the family $\mathcal S$ 
and $X$ is a general fourfold containing $S$, we can build an explicit birational map
$X\stackrel{\simeq}{\dashrightarrow} \PP^4$,
defined over the same field of definition as $S$ and $X$.

\subsection{Cubic fourfolds containing a quintic del Pezzo surface}\label{dp5}

A quintic del Pezzo surface is the image of 
$\PP^2$ via the linear system of cubic curves with 
 $4$ base points in general position.
 
 \begin{theorem}[\cite{Fano,BRS}]
  The cubic fourfolds containing a quintic del Pezzo surface 
describe the divisor 
$\mathcal C_{14}\subset \mathcal C$
of fourfolds of discriminant $14$.  
 \end{theorem}

 \begin{theorem}[\cite{Morin,Fano,BRS}]
 A quintic del Pezzo surface $S\subset\PP^5$
 admits a \emph{conguence of secant lines}: through the general point of $\PP^5$ there passes 
  a unique secant line to $S$.

   The general line of this congruence can be realized as the general 
 fiber of the dominant map 
  \[\mu:\PP^5\dashrightarrow \PP^4 \]
 defined by the linear system $|H^0(\mathcal I_{S}(2))|$ 
 of quadric hypersurfaces through~$S$.
 
If $X$ is a general cubic fourfold containing $S$, then the restriction 
of $\mu$ induces a birational map $\mu|_X:X\stackrel{\simeq}{\dashrightarrow}\PP^4$. 

\end{theorem}

\subsection{Cubic fourfolds containing a $3$-nodal septic scroll}\label{7scroll}
 Let $S\subset\PP^5$ be the projection 
 of a rational  normal septic scroll $\Sigma_7\subset\PP^8$ 
 from a plane spanned by three general points on the secant variety of $\Sigma_7$. 
 Thus $S$ is a rational septic scroll having $3$ non-normal nodes.
 \begin{theorem}[\cite{FV}]
  The family of rational $3$-nodal septic scrolls, constructed as above,
  is irreducible, unirational and of dimension $44=74+5\cdot 3 - \dim\mathrm{PGL}_9 + \dim\mathrm{PGL}_6$.
  
  The cubic fourfolds containing such a surface 
describe the divisor $\mathcal C_{26}\subset \mathcal C$ of fourfolds of discriminant $26$.
 \end{theorem}

 \begin{theorem}[\cite{RS3,procRS}]
  Let $S\subset\PP^5$ be a general rational $3$-nodal septic scroll.
   Then $S$
  admits a \emph{conguence of $5$-secant conics}: through the general point of $\PP^5$ there passes 
  a unique conic curve which is $5$-secant to $S$.

  The general conic curve of this congruence can be realized as the general 
 fiber of the dominant map 
  \[\mu:\PP^5\dashrightarrow \PP^4 \]
 defined by the linear system $|H^0(\mathcal I_{S}^{2}(5))|$ 
 of quintic hypersurfaces with double points along~$S$.
 
 If $X$ is a general cubic fourfold containing $S$, then the restriction 
of $\mu$ induces a birational map $\mu|_X:X\stackrel{\simeq}{\dashrightarrow}\PP^4$.
 \end{theorem}

\subsection{Cubic fourfolds containing a ``generalized'' Coble surface}\label{coble}
Let $S\subset\PP^5$ be the image of $\PP^2$ via the linear system of curves of degree $10$ with 
 $10$ general triple points. We have that $S$ is a smooth rational surface of degree $10$ 
 and sectional genus $6$ cut out by $10$ cubics. 
\begin{theorem}[\cite{Nuer}]
The surfaces $S\subset \mathbb{P}^5$ 
obtained as above
form an irreducible unirational 
family $\mathcal S^{10,6}\subset\mathrm{Hilb}_{\PP^5}$ of dimension $47 = 10\cdot 2 - \dim \mathrm{PGL}_3 + \dim \mathrm{PGL}_6$.

The cubic fourfolds containing a surface of the family $\mathcal S^{10,6}$ 
describe the divisor 
$\mathcal C_{38}\subset \mathcal C$ of fourfolds  of discriminant $38$.
\end{theorem}

 \begin{theorem}[\cite{RS1,RS3,procRS}]
 A general surface $[S]\in \mathcal S^{10,6}$ 
  admits a \emph{conguence of $5$-secant conics}: through the general point of $\PP^5$ there passes 
  a unique conic curves which is $5$-secant to $S$.

  The general conic curve of this congruence can be realized as the general 
 fiber of the dominant map 
  \[\mu:\PP^5\dashrightarrow \PP^4 \]
 defined by the linear system $|H^0(\mathcal I_{S}^{2}(5))|$ 
 of quintic hypersurfaces with double points along~$S$.
 
 If $X$ is a general cubic fourfold containing $S$, then the restriction 
of $\mu$ induces a birational map $\mu|_X:X\stackrel{\simeq}{\dashrightarrow}\PP^4$
 \end{theorem}
 
\subsection{GM fourfolds of discriminant $10$}\label{disc10}

\subsubsection{$\tau$-quadric surfaces}\label{disc101} 
A \emph{$\tau$-quadric} surface is a two-dimensional linear section 
of a Schubert variety $\Sigma_{1,1}\simeq \mathbb{G}(1,3)\subset\mathbb{G}(1,4)$.
Thus the class of such a surface in $\mathbb{G}(1,4)$ is $\sigma_1^2\cdot \sigma_{1,1}=\sigma_{3,1}+\sigma_{2,2}$.
A standard parameter count (see \cite[Proposition~7.4]{DIM}, and also \cite{famGushelMukai})
shows that 
the closure  inside $\MM$ of the family 
of fourfolds containing  
a $\tau$-quadric surface 
forms the divisor $(\MM)_{10}'\subset\MM$, one 
of the two irreducible components of
the Noether-Lefschetz locus in $\MM$ parametrizing fourfolds 
of discriminant $10$.
In particular, since the family of $\tau$-quadric surfaces in $\mathbb{G}(1,4)$ is unirational,
we deduce that the divisor $(\MM)_{10}'$ is also unirational.
\begin{theorem}[\cite{DIM}; see also \cite{HoffSta}]
 The projection of a general fourfold $[X]\in(\MM)_{10}'$ 
 containing a $\tau$-quadric surface $S$, from the linear span $\langle S\rangle\simeq\PP^3$ of $S$,
 gives a birational map $X\dashrightarrow \PP^4$. 
\end{theorem}

\subsubsection{Quintic del Pezzo surfaces}\label{disc102}
A quintic del Pezzo surface 
can be realized as a two-dimensional linear section of $\mathbb{G}(1,4)$.
Thus the class of such a surface in $\mathbb{G}(1,4)$ is $\sigma_1^4 = 3\sigma_{3,1}+2\sigma_{2,2}$.
A standard parameter count (see \cite[Proposition~7.7]{DIM}, and also \cite{famGushelMukai})
shows that 
the closure  inside $\MM$ of the family 
of fourfolds containing  
a quintic del Pezzo surface 
forms the divisor $(\MM)_{10}^{''}\subset\MM$; one 
of the two irreducible components of 
the Noether-Lefschetz locus in $\mathcal M$
of fourfolds  of discriminant~$10$.
In particular, since the family of quintic del Pezzo surfaces in $\mathbb{G}(1,4)$ is unirational,
we deduce that the divisor $(\MM)_{10}^{''}$ is also unirational.
\begin{theorem}[\cite{Roth1949}; see also \cite{DIM,Enr,EinSh}] 
 The projection of a general fourfold $[X]\in(\MM)_{10}^{''}$ 
 containing a quintic del Pezzo surface $S$, from the linear span $\langle S\rangle\simeq\PP^5$ of $S$,
 induces a dominant map $X\dashrightarrow\PP^2$ 
 whose generic fiber is a quintic del Pezzo surface.
 (In particular, $X$ is rational. Indeed,
 from a classic result of Enriques,
  a quintic del Pezzo surface 
 defined over an infinite field $K$ is rational over $K$.)
\end{theorem}

\subsection{GM fourfolds of discriminant $20$}\label{42g}
Throughout  this subsection, we continue to let
 $\YY\subset\PP^8$ denote a fixed del Pezzo fivefold. 
\medskip 

Recall first two well-known ways to parametrize $\YY$ over its field of definition.
\begin{enumerate}
\item If $P\subset \YY$ is a plane in $\YY$ with class 
$\sigma_{2,2}$ in $\mathbb{G}(1,4)$, then the projection of $\YY$ from $P$ gives a birational map 
$\YY\dashrightarrow\PP^5$, whose inverse is defined by the linear system of quadrics 
through a rational normal cubic scroll $\Sigma_3\subset\PP^4\subset\PP^5$.
 \item If $C\subset \YY$ is a conic such that its linear span $P$ is not contained in $\YY$, then the projection of 
$\YY$ from $P$ gives a birational map $\YY\dashrightarrow\PP^5$,
whose inverse 
is defined 
by the linear system of cubics through a rational quartic scroll $\Theta_4 \subset\PP^5$,
obtained as a general projection of a rational normal threefold scroll in $\PP^6$.
\end{enumerate}
\medskip 

Now we breafly recall the construction due to \cite{RS3,HoffSta}
of a $25$-dimensional unirational family $\mathcal S^{9,2}\subset\mathrm{Hilb}_{\YY}$
of smooth rational surfaces of degree $9$ and genus $2$ 
having class $6\sigma_{3,1}+3\sigma_{2,2}$ in the Chow ring of $\mathbb{G}(1,4)$.
\medskip 

Let $T\subset\PP^6$ be the image of the plane 
via the linear system of quartic curves having $8$ general base points $p_1,\ldots,p_8$.
Thus $T$ is a smooth rational surface 
of degree $8$ and sectional genus $3$ 
cut out by $7$ quadrics. 
These $7$ quadrics
define a \emph{special} Cremona transformation 
\[
 \varphi:\PP^6\dashrightarrow\PP^6
\]
of type $(2,4)$, 
which has been classically studied in \cite{cite2-semple} (see also \cite{hulek-katz-schreyer}).
\medskip 

Let us recall a bit of geometry from the papers \cite{cite2-semple,hulek-katz-schreyer}.
The pencil of plane cubics through the $8$ 
base points $p_1,\ldots,p_8$
yields 
a pencil of elliptic normal quartic curves on $T$ 
passing through a special point $q\in T$, and
the union of the linear spans of these curves
gives a cone of vertex $q$ over a Segre threefold $\PP^1\times\PP^2\subset\PP^5$.

Let $H\simeq\PP^5\subset\PP^6$ be a general hyperplane in $\PP^6$.  
The restriction of $\varphi$ to $H$ gives 
a birational map 
\[\varphi|_H:\PP^5\dashrightarrow Z\subset\PP^6\]
onto a quartic hypersurface $Z\subset\PP^6$,
whose base locus, that is the intersection 
of $T$ with the hyperplane $H$, is an octic curve $C\subset\PP^5$
of arithmetic genus $3$ contained in a Segre threefold $\Sigma\simeq\PP^1\times\PP^2\subset\PP^5$.
The image  $\overline{\varphi(\Sigma)}$ 
is a smooth quadric surface $Q\subset Z\subset\PP^6$ (which is double as a component of the base locus of the inverse of $\varphi|_H$).
The pullback via the restriction $\varphi|_H$
of a line in one of the two pencils of lines on $Q$ is
a $\PP^2$ of the ruling of $\Sigma$, 
while the pullback of a general line in the other pencil of lines on $Q$
is  a smooth quintic del Pezzo surface containing $C$.
In particular, the curve $C$ is the base locus 
of a pencil  $\{D_{\lambda}\}_{\lambda}$ of quintic del Pezzo surfaces contained in $\Sigma$ and
whose general member is smooth.

Everything we have said about $C$ 
continues to hold true even if we take the hyperplane 
$H\subset \PP^6$ to be general among the hyperplanes containing a general tangent plane to $T$.
But in this case  (and only in this case), the curve $C$ 
has a node and it can be embedded in a rational quartic scroll 
$\Theta_4\subset\PP^5$ as the one considered above.
Indeed, such a nodal curve $C$ can be realized as a nodal 
projection of a smooth curve of degree $8$ and genus $2$ 
contained in a smooth rational normal quartic scroll threefold in $\PP^6$;
see \cite[Section~4]{RS3} for more details on this last step.
Then the birational map 
$\PP^5\dashrightarrow \YY\subset\PP^8$ 
defined by the linear system of cubics through $\Theta_4$ 
induces an isomorphism between a general quintic del Pezzo surface of the pencil $\{D_{\lambda}\}_{\lambda}$
with a smooth rational surface 
$S\subset \YY\subset\PP^8$ of degree $9$ and sectional genus $2$,
cut out by $19$ quadrics, and having class $6\sigma_{3,1}+3\sigma_{2,2}$ in
$\mathbb{G}(1,4)$.
\begin{theorem}[\cite{RS3}, see also \cite{HoffSta}]
The surfaces $S\subset \YY$ 
produced by the construction above
form an irreducible unirational 
family $\mathcal S^{9,2}\subset\mathrm{Hilb}_{\YY}$ of dimension $25$.

The closure of the family of quadric hypersurfaces in $\YY$ containing a surface of the family $\mathcal S^{9,2}$,
after passing to the quotient modulo $\mathrm{PGL}_9$,
describes the divisor $(\MM)_{20}\subset \MM$,
the irreducible component of 
the Noether-Lefschetz locus in $\MM$
of fourfolds  of discrimiant~$20$.
\end{theorem}

\begin{remark}
An implementation of the construction of the family $\mathcal S^{9,2}$ is provided by
 the \emph{Macaulay2} package \emph{SpecialFanoFourfolds} \cite{macaulay2,SpecialFanoFourfoldsSource}.
 In particular, one is able to find explicit equations of a 
 general member of the family.
\end{remark}

\begin{remark}[\cite{famGushelMukai}]
Let $\Sigma_3\subset\PP^4\subset\PP^5$ be a rational cubic scroll surface,
and let $\psi:\PP^5\dashrightarrow \YY\subset \PP^8$
be the birational map defined by the quadrics through $\Sigma_3$.
Take $D\subset\PP^5$ to be a quintic del Pezzo surface intersecting $\Sigma_3$
along a hyperplane section of $\Sigma_3$. 
 Then the restriction of $\psi$ induces an isomorphism between 
 $D$ and a surface $S$ belonging to the family $\mathcal{S}^{9,2}$.
 However, the surfaces $S$ obtained by this ``\emph{simplified}'' construction 
 do not describe the whole family 
 $\mathcal{S}^{9,2}$.
\end{remark}

\begin{theorem}[\cite{HoffSta}]
 Let $S\subset \YY$ be a surface corresponding to a general member of the family $\mathcal S^{9,2}$.
 Then $S$ admits inside $\YY$ a \emph{congruence of $3$-secant conic curves},
 that is, through the general point of $\YY$ there passes 
 a unique conic which is $3$-secant to $S$ and is contained in $\YY$.

   The general conic curve of this congruence can be realized as the general 
 fiber of the dominant map 
  \[\mu:\YY\dashrightarrow \PP^4 \]
 defined by the linear system $|H^0(\mathcal I_{S,\YY}^{2}(5))|$ 
 of quintic hypersurfaces in $\YY$ with double points along~$S$.
 
 If $X$ is a general quadric hypersurface in $\YY$ containing $S$, then the restriction 
of $\mu$ induces a birational map $\mu|_X:X\stackrel{\simeq}{\dashrightarrow}\PP^4$.
\end{theorem}

\begin{remark}\label{K320}
 The above theorem suggests an alternative construction 
 for the general surface $S\subset \YY$ of the family $\mathcal S^{9,2}$,
 which historically was the first to be discovered \cite{HoffSta}.
 Indeed, 
 the inverse map of $\mu|_X:X\stackrel{\simeq}{\dashrightarrow}\PP^4$ 
 is 
 defined by the linear system $|H^0(\mathcal I_{U,\PP^4}^2(9)|$
of hypersurfaces of degree $9$ singular 
along a surface $U$ which is obtained as an internal projection of a triple projection 
of a minimal K3 surface of degree $20$ and genus $11$ in $\PP^{11}$.
We can reverse this construction by starting with a general 
K3 surface of degree $20$ and genus $11$ in $\PP^{11}$.
\end{remark}

\subsection{Cubic fourfolds of discriminant $42$}\label{42c}
The 
 $25$-dimensional family $\mathcal S^{9,2}\subset\mathrm{Hilb}_{\YY}$ considered in the previous subsection and the 
 $3$-dimensional family of planes in $\YY$ with class $\sigma_{2,2}$
 can be combined together 
 to get a family of surfaces in $\PP^5$ of dimension $48=25+3-\dim \mathrm{Aut}(\YY) + \dim \mathrm{Aut}(\PP^5)$.
Indeed, 
let $S\subset \YY$ 
be a surface corresponding to a general member of the family $\mathcal S^{9,2}$,
and let $P\subset \YY$ be a general plane with class $\sigma_{2,2}$ in $\mathbb{G}(1,4)$.
Then the projection of $S$ from $P$ gives 
a rational surface $\widetilde{S}\subset \PP^5$ of degree $9$ and sectional genus $2$, cut 
out by $9$ cubics and having $5$ non-normal nodes.
Let us denote by $\widetilde{\mathcal{S}^{9,2}}\subset\mathrm{Hilb}_{\PP^5}$ the
(closure of the) family of surfaces $\widetilde{S}$ obtained as above. 
\begin{theorem}[\cite{RS3}]
The family $\widetilde{\mathcal S^{9,2}}\subset\mathrm{Hilb}_{\PP^5}$ 
is irreducible and unirational 
 of dimension~$48$.

The cubic fourfolds containing a surface of the family $\widetilde{\mathcal S^{9,2}}$ 
describe the divisor 
$\mathcal C_{42}\subset \mathcal C$ of fourfolds of discriminant $42$.
\end{theorem}
\begin{theorem}[\cite{RS3}]
 Let $\widetilde{S}\subset \PP^5$ be a surface corresponding to a general member of the family $\widetilde{\mathcal S^{9,2}}$.
 Then $\widetilde{S}$ admits  a \emph{congruence of $8$-secant twisted cubic curves},
 that is, through the general point of $\PP^5$ there passes 
 a unique twisted cubic which is $8$-secant to~$\widetilde{S}$.

   The general cubic curve of this congruence can be realized as the general 
 fiber of the dominant map 
  \[\mu:\PP^5\dashrightarrow W \subset\PP^7\]
  onto a smooth del Pezzo fourfold $W=\mathbb{G}(1,4)\cap\PP^7$,
 defined by the linear system $|H^0(\mathcal I_{\widetilde{S},\PP^5}^{3}(8))|$ 
 of octic hypersurfaces with triple points along~$\widetilde{S}$.
 
 If $X$ is a general cubic fourfold containing $S$, then the restriction 
of $\mu$ induces a birational map $\mu|_X:X\stackrel{\simeq}{\dashrightarrow} W$.
\end{theorem}

\begin{remark}
It follows from well-known classic results that a del Pezzo fourfold $W=\mathbb{G}(1,4)\cap\PP^7$ defined over an infinite field $K$ is rational over $K$.
\end{remark}

\section{Main Theorem}\label{mainthm} 
The goal of this Section is to prove the following result.

\begin{theorem}\label{newmain}
The moduli spaces of $n$-pointed fourfolds over the following moduli loci are unirational:

\begin{enumerate}
\item $\C_{14}$: cubic fourfolds containing a  quintic del Pezzo surface (Subsection~\ref{dp5});
\item $\C_{26}$: cubic fourfolds containing a 3-nodal septimic scroll (Subsection~\ref{7scroll});
\item $\C_{38}$: cubic fourfolds containing a generalized Coble surface (Subsection~\ref{coble});
\item  $\C_{42}$: cubic fourfolds containing a $5$-nodal
 rational surface of degree 9 and sectional genus 2 (Subsection~\ref{42c});
 \item $(\MM)_{10}'$: GM fourfolds containing a $\tau$-quadric surface (Subsection~\ref{disc101});
 \item $(\MM)_{10}^{''}$: GM fourfolds containing a quintic del Pezzo surface (Subsection~\ref{disc102});
\item $(\MM)_{20}$: GM fourfolds containing a smooth rational surface of degree 9 and sectional genus 2 (Subsection~\ref{42g});
\end{enumerate}
\end{theorem}

\begin{remark}\label{properties}
Recall that we denoted by $\mathcal U\subset H^0(\mathcal{O}_{\PP^5}(3))$ (respectively, $\mathcal V\subset H^0( \mathcal{O}_{\mathbb{Y}^5}(2))$) the open 
set parametrizing smooth cubic hypersurface in $\PP^5$ (respectively, smooth quadric hypersurfaces in $\YY$).
If $\mathcal S$ is a family of surfaces in the the Hilbert scheme $\mathrm{Hilb}_{\PP^5}$ of $\PP^5$ (respectively,
in the Hilbert scheme $\mathrm{Hilb}_{\YY}$ of  $\YY$),
then we denote by ${\mathcal X}_{\mathcal {S}}$ 
the closure inside $\mathcal U$ (respectively, inside $\mathcal V$) of the family of fourfolds  that contain some surface $[S]\in \mathcal S$,
and we let $\widetilde{{\mathcal X}_{\mathcal {S}}} = {\mathcal X}_{\mathcal {S}}//\mathrm{Aut}(\PP^5)\subseteq \mathcal C$ (respectively, $\widetilde{{\mathcal X}_{\mathcal {S}}} = \overline{{\mathcal X}_{\mathcal {S}}//\mathrm{Aut}(\YY)}\subseteq  \MM).$
We observe that the families of fourfolds described in Section~2 (and object of Theorem~\ref{newmain}) all share the following properties:
\begin{enumerate}
\item $\mathcal S$ is irreducible and unirational; so that 
the same holds true for 
the corresponding family $\mathcal X_{\mathcal S}$, and hence  for $\widetilde{\mathcal X_{\mathcal S}}$.
 \item\label{ipotesi vera} If $(S,X)$ is a couple where 
$S$ is a general member of the family $\mathcal S$ 
and $X$ is a general fourfold containing $S$,
then we are able to build,  starting from the pair $(S,X)$, 
an explicit birational map $\psi_{(S,X)}:\mathbb{P}^4\stackrel{\simeq}{\dashrightarrow} X$,
defined over the same field of definition as $S$ and $X$.
\end{enumerate}

\end{remark}

\newcommand{\XX}{\mathcal{X}}
\newcommand{\cS}{\mathcal{S}}
\newcommand{\cY}{\mathbb{Y}}

\begin{remark}\label{families}
The family $\XX_\cS$ carries the universal 1-pointed fourfold $\XX_{\cS, 1} \to \mathcal{X}_\cS$. And inductively one can define a tower of maps

$$\dots \to \mathcal{X}_{\cS, n} \to \mathcal{X}_{\cS, n-1}\to\dots \to\mathcal{X}_{\cS, 1}\to \mathcal{X_\cS}.$$
\end{remark}

By quotienting out by the automorphisms of $\PP^5$ or $\mathbb{Y}^5$, we can give straight away the following definition.

\begin{definition}\label{pointedmoduli}
We will denote by $\widetilde{\mathcal{X}_{\cS, n}}$ the moduli space of $n$-pointed (cubic or GM) fourfolds. It is the quotient by the respective group of automorphisms of the family $ \mathcal{X}_{\cS, n}$.
\end{definition}

\begin{remark}
The existence over an open subset of $\widetilde{\XX_\cS}$ of the moduli spaces in Def. \ref{pointedmoduli} is guaranteed by the fact that the very general cubic fourfold and GM fourfold has no nontrivial automorphisms (\cite[Theorem~3.8]{autom}, \cite[Proposition 3.21]{automord}). In fact, a family of cubic fourfolds with at least a non-trivial projective automorphism has dimension at most 14 \cite[Theorem~3.8]{autom}. Since we are working in the birational category, this will be enough for our goals.
\end{remark}

Of course there are forgetful maps

$$\widetilde{\mathcal{X}_{\cS, n}} \stackrel{\pi_n}{\to} \widetilde{\mathcal{X}_{\cS, n-1}} \stackrel{\pi_{n-1}}{\to}\widetilde{\mathcal{X}_{\cS, n-2}}\to \cdots ,$$
with evident meaning.

\begin{proof}[Proof of Theorem~\ref{newmain}]
Let us consider the unirational family $\mathcal{S}$ of surfaces inside $\PP^5$ (resp., $\YY$). 
Over the function field of $\mathcal{S}$, we can write the equations of the surface $S \in \mathcal{S}$, 
and notably describe their ideal. 
Call $m$ the dimension of  $H^0(\PP^5, \mathcal{I}_S(3))$
(resp.,  $H^0(\cY^5,\mathcal{I}_S(2))$). This is generically constant over $\mathcal{S}$. This means that the Hilbert scheme of couples $(S,X)$, 
where $X$ is a fourfold containing the surface $S\in \cS$, is birational to a $\PP^{m-1}-$bundle over $\mathcal{S}$, and hence unirational. Let us denote it by $P_{(S,X)}$. This sits naturally inside the product 
$\mathcal{S}\times \mathcal{X}$, 
 where $\mathcal{X}$ is the moduli space of cubic (resp., GM) fourfolds, and has two natural projections to the two components.

Now we observe that, by definition, over $P_{(S,X)}$ there is a natural \it double \rm universal family. That is: since $P_{(S,X)}$ parametrizes the couples $(S,X)$, then on one side we have the universal family $\mathcal{S}_1$ of 1-pointed surfaces - that is the universal family of surfaces over $\mathcal{S}$, pulled-back to $P_{(S,X)}$, on the other we have the universal family $\mathcal{X}_{\cS,1}$ 
of 1-pointed cubic (resp., GM) fourfolds with the forgetful map $\mathcal{X}_{\cS,1} \xrightarrow{\pi} P_{(S,X)}$. 
By definition there is a fiberwise inclusion: 
 \begin{equation*}
\xymatrix{ \mathcal{S}_1 \ar@{^{(}->}[rr] \ar[dr] & & \mathcal{X}_{\cS,1} \ar[dl] \\
 & P_{(S,X)} &  \\}
 \end{equation*}

Remark that different points of $P_{(S,X)}$ may parametrize the same GM fourfold but a different surface.

Now, by assumption $(2)$ of our working hypotheses at the beginning of this section, we have that - over $P_{(S,X)}$ - we can define a relative linear system on $\mathcal{X}_{\cS,1}$ (with base locus supported on $\mathcal{S}_1$) that defines a birational map from $\mathcal{X}_{\cS,1}$ to a $\PP^4$-bundle over $P_{(S,X)}$, that we denote $\PP^4_{P_{(S,X)}}$. Now $\PP^4_{P_{(S,X)}}$ is rational over $P_{(S,X)}$ and $P_{(S,X)}$ is unirational, hence the universal family $\mathcal{X}_{\cS,1}$ is unirational, since it is birational to $\PP^4_{P_{(S,X)}}$.

Exactly as one does for $\mathcal{X}_{\cS,1}$, we can construct a universal cubic (resp., GM) fourfold over $\mathcal{X}_{\cS,1}$, just by taking the pull-back $\pi^* \mathcal{X}_{\cS,1}$ over $\mathcal{X}_{\cS,1}$. We denote by $\mathcal{X}_{\cS,2}$ this family, and we observe that it tautologically contains $\pi^*\mathcal{S}_1$, as the following diagram follows.

$$\xymatrix{ \pi^*\mathcal{S}_1 \ar@{^{(}->}[r] \ar[d] & \mathcal{X}_{\cS,2}= \pi^* \mathcal{X}_{\cS,1} \ar[d] \\ 
\mathcal{S}_1 \ar@{^{(}->}[r] &  \mathcal{X}_{\cS,1} \\
}$$

 Thus $\mathcal{X}_{\cS,2}$ has the same property $(2)$ as $\mathcal{X}_{\cS,1}$ and one can define a relative linear system defining the birationality between $\mathcal{X}_{\cS,2}$ and a $\PP^4$-bundle over $\mathcal{X}_{\cS,1}$ - that we denote by $\PP^4_{\mathcal{X}_{\cS,1}}$. By the same argument as above, since $\mathcal{X}_{\cS,1}$ is unirational, $\mathcal{X}_{\cS,2}$ is unirational as well. Then, inductively, the same argument shows the unirationality of the universal families $\mathcal{X}_{\cS,n}$, for all $n$. 
 
Now the natural classifying maps, given by the quotient by the automorphisms groups, make the following diagram commutes.

 \begin{equation*}
\xymatrix{ \pi^*\mathcal{S}_1 \ar@{^{(}->}[r] \ar[dd] & \mathcal{X}_{\cS,2} = \pi^*\mathcal{X}_{\cS,1}\ar[dr]^\sim \ar[rr]^{//Aut}\ar[dd] & & \widetilde{\mathcal{X}_{\cS,2}} \ar[dd]  \\ 
& & \PP^4_{\mathcal{X}_{\mathcal{S},1}}\ar[dl] \ar[ru] &  \\
\mathcal{S}_1 \ar@{^{(}->}[r] \ar[ddr] \ar[dd]& \mathcal{X}_{\cS,1} \ar[dd]^\pi \ar[dr]^\sim \ar[rr]^{//Aut}& & \widetilde{\mathcal{X}_{\cS,1}} \ar[dd] \\
& & \PP^4_{P_{(S,X)}} \ar[dl] \ar[ru] & \\
\mathcal{S} & P_{(S,X)} \ar[l] \ar[rr]^{//Aut} & & \widetilde{\mathcal{X}_{\cS}}\\
} 
 \end{equation*}

It is then clear that also the corresponding moduli space $\widetilde{\XX_{\cS,n}}$, corresponding to the families $\XX_{\cS,n}$ are unirational.

\smallskip

To conclude the proof we observe that, thanks to the properties (1) and (2),
 we can plug any one of the seven loci mentioned in the claim inside this construction, and get the result.
\end{proof}

\begin{remark}
Several other special families of cubic (resp., GM) fourfolds verify the hypotheses required in this section. We nevertheless decided to concentrate on certain particular descriptions of codimension one loci.

In fact we did not only choose some codimension 1 loci, 
but we also chose a particular description of them. For example, for cubic fourfolds in $\C_{14}$ we could have chosen quartic scrolls as surfaces defining the divisor. We remark however that in that case our argument would not have worked 
since the quartic scroll defines a birational map to a 4-dimensional quadric, and a quadric bundle is not automatically rational over its base.
\end{remark}

\section{A rational Noether-Lefschetz divisor of genus 11 K3 surfaces, and their associated GM fourfolds.}\label{g11}

In this section we will consider a codimension 1 locus of $(\mathcal{M}^4_{GM})_{20}$, in order to show that, once we restrict the family of GM-fourfolds to this locus, stronger rationality statements hold.





As observed in Rem. \ref{K320} (see also \cite{HoffSta,RS3}), for the generic $X\in (\mathcal{M}^4_{GM})_{20}$ there exists a birational map $\xymatrix{\PP^4 \ar@{-->}[r] & X}$, defined  by the linear system of hypersurfaces of degree 9 having double points along a surface $U$, which is a projection of a genus 11 K3 surface.

\smallskip

More precisely, one starts from a K3 surface $Z\subset \PP^{11}$ of degree 20 and sectional genus 11. We take two points $p,q\in Z$, and perform first a triple projection from $p$ to $\PP^5$, then a simple projection off $q$ to $\PP^4$. The image is the required surface $U$.


With this in mind, one can prove the following result (that we have already proven), in a new fashion. We give a sketch of this different proof since it will be useful in the following.

\begin{theorem}\label{theo}
The universal family $(\mathcal{M}^4_{GM})_{20,1}$ of 1-pointed GM-fourfolds is unirational.
\end{theorem}

\begin{proof}
The philosophy is to do the above rationality construction in families. We need then to consider the moduli space $\mathcal{F}_{11,2}$ of polarized K3 surface of genus 11 with two marked points. 
The moduli space $\mathcal{F}_{11,3}$ 
comes equipped with an embedding inside a $\PP^{11}$-bundle over $\mathcal{F}_{11,2}$ and with two sections $\delta_1,\delta_2: \mathcal{F}_{11,2}\to \mathcal{F}_{11,3}$. Performing a relative triple projection from the image of $\delta_1$ and a simple one from the image of $\delta_2$ we obtain a $\PP^4$-bundle $\PP(E)$ over $\mathcal{F}_{11,2}$ containing the family $U$ of degree 10 surfaces.

    $$ \xymatrix @!0 @R=5mm @C=4cm {\relax
   \mathcal{F}_{11,3} \subset \PP^{11}\ar@{-->}[r]  \ar[ddddddd]  &  U \subset  \PP(E) \ar[ddddddd] \\
  &\\
  &\\
    &\\
    &\\
    &\\
    &\\
    \mathcal{F}_{11,2} \ar@/^1pc/[uuuuuuu]^{\delta_1}\ar@/_1pc/[uuuuuuu]_{\delta_2} & \mathcal{F}_{11,2} }$$

Since $\mathcal{F}_{11,2}$ is unirational \cite[Theorem 0.1]{Barros_2018} , the projective bundle $\PP(E)$ is unirational. The relative linear system of degree 9 hypersurfaces, with multiplicity two along $U$ gives a rational dominant map between the $\PP^4$-bundle $\PP(E)$ and $(\mathcal{M}^4_{GM})_{20,1}$ 
\end{proof}

In \cite{HoffSta}, Hoff and Staglian\`o also consider a codimension one subfamily of genus 11 K3 surfaces, that forms a Noether-Lefschetz divisor inside $\mathcal{F}_{11}$. This divisor seems particularly interesting under our point of view, since the wealth of geometry going on here allows us to strenghten our rationality results concerning the corresponding universal families of GM fourfolds related to these K3 surfaces. But let us give a couple more details about these surfaces.

We start from a smooth Fano threefold $Y$ of type $X_{22}\subset \PP^{13}$. It is well known that the generic tangent hyperplane sections of $Y$ are one-nodal (a double point) K3 surfaces (see \cite{{mukai-biregularclassification}}, \cite{schr}). The projection off the node of such a K3 surface gives a K3 surface in $\PP^{11}$, of degree 20 and sectional genus 11, containing a further conic (the exceptional divisor over the node). In fact, such a construction gives a Noether-Lefschetz divisor inside the 19-dimensional moduli space of K3 surfaces of genus 11, 
and the intersection lattice of these surfaces contains a sublattice of type

$$
    \begin{pmatrix} 20 & 2 \\ 2 & -2  \end{pmatrix}.
$$ 

\medskip


Before studying the universal family of GM fourfolds obtained from these special K3 surfaces, we need to show some results on the birational geometry of their Noether-Lefschetz locus. We will denote by $\mathcal{V}^{nod}_n$ the moduli space of $n-$pointed one nodal K3 surfaces of  sectional genus 12, obtained by cutting a $X_{22}$-type 3fold with tangent hyperplanes as above. The generic element of $\Vn$ is represented by a vector
$(Y,p,H,q_1,\dots, q_n)$, where $Y$ is a Fano threefold of type $X_{22}$, $p$ is a point of $Y$, $H$ is a hyperplane tangent to $Y$ in $p$, and $q_1,\dots, q_n$ are $n$ points on the surface $S_H:=Y\cap H$.
We will also denote by $\mathcal{X}_{22}$ the (rational, see \cite{Mu}) moduli space of Fano threefolds of type $X_{22}$. All these Fano threefolds are rational and birational among them, we will need to fix one $\widetilde{X}_{22}\in \mathcal{X}_{22}$.

\begin{theorem}

The moduli space $\Vn$ is rational if $n\leq 9$.

\end{theorem}
     
\begin{proof}
Let us consider the rational map

\begin{eqnarray}
\varphi: \Vn &\to & \mathcal{X}_{22} \times \widetilde{X}_{22}^{n+1}\\
(Y,p,H,q_1,\dots ,q_n) & \mapsto & (Y,p,q_1,\dots, q_n).
\end{eqnarray}

Remark that $\mathcal{X}_{22}\times \widetilde{X}_{22}^{n+1}$ is rational (and of dimension $3n+ 9$) since it is the product of rational varieties. Then, the fiber of $\varphi$ over $(Y,p,q_1,\dots, q_n)$ is exactly the linear system of hyperplanes in $\PP^{13}$ that are tangent to $Y$ in $p$, and pass through $q_1,\dots,q_n$. This shows that $\Vn$ is birational to a $\PP^{9-n}-$projective bundle over $\mathcal{X}_{22}\times \widetilde{X}_{22}^{n+1}$, and hence is rational if $n\leq 9$.
\end{proof}     

We recall that the projection off the node sends birationally $\Vn$ onto a $(18+2n)$-dimensional NL locus inside $\mathcal{F}_{11,n}$. Let us denote by $(\mathcal{M}^4_{GM})_{20}^{nod}$, the moduli space of GM fourfolds obtained from the NL K3 surfaces described above, and by $(\mathcal{M}^4_{GM})_{20,1}^{nod}$ the universal family above, obtained by restricting the construction of Thm. \ref{theo}. The moduli space $(\mathcal{M}^4_{GM})_{20}^{nod}$ is of dimension 22, and is contained in $(\mathcal{M}^4_{GM})_{20}$.
     
\begin{corollary}
The universal family $(\mathcal{M}^4_{GM})_{20,1}^{nod}$ is rational. The moduli space $(\mathcal{M}^4_{GM})_{20}^{nod}$ is rational.
\end{corollary}

\begin{proof}
The moduli space $\Vtre$ of nodal, 3-pointed K3 surfaces can be embedded in a $\PP^{12}$-bundle, and endowed with two sections $\delta_1,\delta_2: \Vdue \to \Vtre$, over $\Vdue$. Since we are working in the birational category, we can even consider (at least an open subset of) $\Vdue$ as contained in $\mathcal{F}_{11,2}$. Now, we project fiberwise off the node, obtaining a family of NL K3 surfaces in a $\PP^{11}$-bundle, with two sections, over $\Vdue$. Then, as we did in Thm. \ref{theo}, we perform the two projections off the sections and we obtain a $\PP^4$-bundle over $\Vdue$, containing a family $\mathcal{T}$ of degree 10 surfaces. The moduli space $\Vdue$ is rational, hence the $\PP^4$-bundle is rational as well. Then, by applying the relative linear system of degree 9 hypersurfaces through $\mathcal{T}$ as in Theorem \ref{theo}, we obtain a rational family of GM fourfolds over $\Vdue$, hence $(\mathcal{M}^4_{GM})_{20,1}^{nod}$ is rational. By construction $(\mathcal{M}^4_{GM})_{20}^{nod}$ is birational to $\Vdue$ and hence rational. 
\end{proof}


\providecommand{\bysame}{\leavevmode\hbox to3em{\hrulefill}\thinspace}
\providecommand{\MR}{\relax\ifhmode\unskip\space\fi MR }
\providecommand{\MRhref}[2]{%
  \href{http://www.ams.org/mathscinet-getitem?mr=#1}{#2}
}
\providecommand{\href}[2]{#2}

\end{document}